\documentclass{amsart}
\usepackage{amsfonts}
\usepackage{amsmath,amssymb}
\usepackage{amsthm}
\usepackage{amscd}
\usepackage{graphics}
\usepackage{graphicx}

\theoremstyle{remark}{
\newtheorem{Def}{{\rm Definition}}
\newtheorem{Ex}{{\rm Example}}
\newtheorem{Rem}{{\rm Remark}}
\newtheorem{Prob}{{\rm Problem}}

}
\theoremstyle{plain}
{
\newtheorem{Cor}{Corollary}
\newtheorem{Prop}{Proposition}
\newtheorem{Thm}{Theorem}

}

\begin{document}
\title[A note on shapes associated to two smooth real-valued functions]{A note on asymptotic behaviors and topological properties of two smooth real-valued functions and several graphs associated to them}
\author{Naoki kitazawa}

\address{Osaka Central Advanced Mathematical Institute (OCAMI) \\
3-3-138 Sugimoto, Sumiyoshi-ku Osaka 558-8585
TEL: +81-6-6605-3103
}
\email{naokikitazawa.formath@gmail.com}
\urladdr{https://naokikitazawa.github.io/NaokiKitazawa.html}
\maketitle
\begin{abstract}

This is a note on the graphs of two smooth real-valued functions in the plane with no intersection and the natural map onto the region surrounded by them with the canonical projection to the line composed, yielding its {\it Reeb space}. 

The Reeb space of a real-valued function on a topological space is the set of all connected components of all level sets and topologized naturally. Such spaces have been fundamental and strong tools in theory of Morse functions and its variants, since the former half of the 20th century. They are graphs for tame functions such as Morse(-Bott) functions.

The author has launched and studied the problem above, since the 2020s. He is interested in Reeb spaces of smooth or real analytic non-proper functions.
For closed manifolds and nice compact spaces, topological properties and combinatorial ones on Reeb spaces have been investigated by Gelbukh, Saeki, and so on.

\end{abstract}
%【REVISE】 combinatoric ～ is → combinatorial object. It is .
%【REVISE】  such that a point is a vertex if and only if the corresponding connected component of the level set contains some singular points → whose vertex set is the set of all points containing some singular points in the corresponding connected component of the level set .
%【REVISE】 We delete "extending the result before".
\section{Introduction.}
\label{sec:1}
This is a note on the following problem, started in \cite{kitazawa9}.
\begin{Prob}
\label{prob:1}
Let $c_i:\mathbb{R} \rightarrow \mathbb{R}$ ($i=1,2$) be real-valued smooth functions in the real plane ${\mathbb{R}}^2$ such that the graphs $\{(c_1(x),x) \mid x \in \mathbb{R}\}$ and $\{(c_2(x),x) \mid x \in \mathbb{R}\}$ do not intersect. Let $D_{c_1,c_2}:=\{(x_1,x_2) \mid c_1(x_2)<x_1<c_2(x_2)\}$. For each integer $m \geq 2$ we have an $m$-dimensional smooth submanifold $X_{m,c_1,c_2}:=\{(x_1,x_2,(y_j)_{j=1}^{m-1})\mid  (x_1-c_1(x_2))(c_2(x_2)-x_1)-{\Sigma}_{j=1}^{m-1} {y_j}^2=0\}$ with no boundary of the ($m+1$)-dimensional Euclidean space ${\mathbb{R}}^{m+1}$. By restricting the canonical projection ${\pi}_{m+1,2}$ to the $m$-dimensional submanifold, we have a smooth map onto the closure $\overline{D_{c_1,c_2}}:=\{(x_1,x_2) \mid c_1(x_2) \leq x_1 \leq c_2(x_2)\}$: let ${\pi}_{k_1+k_2,k_1}:{\mathbb{R}}^{k_1+k_2} \rightarrow {\mathbb{R}}^{k_1}$ denote the canonical projection defined by ${\pi}_{k_1+k_2,k_1}(x_1,x_2):=x_1$. The fact that this is a smooth manifold is reviewed in the next section. We consider the composition of the previous map onto $\overline{D_{c_1,c_2}} \subset {\mathbb{R}}^2$ with ${\pi}_{2,1}$. This function is also represented as ${\pi}_{m+1,1} {\mid}_{X_{m,c_1,c_2}}$, where for a map $c:X \rightarrow Y$ and a subset $Z \subset X$, we use $c {\mid}_Z$ for the restriction of $c$ to $Z$.

Investigate topological properties and combinatorial ones of $c:={\pi}_{m+1,1} {\mid}_{X_{m,c_1,c_2}}:X:=X_{m,c_1,c_2} \rightarrow Y:=\mathbb{R}$.
\end{Prob}

For this, we need the notion of {\it Reeb space} $R_c$ of a continuous function $c:X \rightarrow \mathbb{R}$ on a topological space $X$. Reeb spaces are main tools and objects here.

We can define an equivalence relation ${\sim}_c$ on $X$ by the following rule. For two points $x_1$ and $x_2$, $x_1 {\sim}_c x_2$ if and only if they are in a same connected component ({\it contour}) of a same {\it level set} $c^{-1}(y)$. 
The quotient space $R_c:=X/{{\sim}_c}$ is the {\it Reeb space} of c and we have the quotient map $q_c:X \rightarrow R_c$ and the unique continuous function $\bar{c}$ with $c=\bar{c} \circ q_c$.
This has appeared in \cite{reeb}. Since this, Reeb spaces have been fundamental and strong tools in investigating manifolds by Morse(-Bott) functions and naturally generalized functions.

Reeb spaces are, for such tame functions, regarded as graphs. For the Morse function case see \cite{izar} and the Morse-Bott function case see \cite{martinezalfaromezasarmientooliveira}. General theory on graph structures and rigorous theory on the fact that the Reeb spaces of real-valued continuous functions on nice compact topological spaces are $1$-dimensional and topologically nice are \cite{gelbukh1, gelbukh2, gelbukh3, gelbukh4, saeki1, saeki2}.

For example, realization of finite graphs as Reeb spaces of nice smooth functions on closed manifolds has been fundamental and important problems. Sharko has launched this in 2006 (\cite{sharko}) and Masumoto-and Saeki follow this (\cite{masumotosaeki}). They have reconstructed nice smooth functions on closed surfaces with given Reeb spaces. First, Sharko has considered a certain class of finite graphs and reconstruction of a smooth function $c$ each critical point $p$ of which is of the form $c(z)={\Re}(z^l)+c(p)$ or a critical point of a Morse function: here $z$ is a complex number, $p$ is identified naturally with $0$ by a suitable local coordinate, and $l>1$ is an integer. Masumoto and Saeki have extended his theory for arbitrary graphs and constructed smooth functions whose critical points may not be isolated. Michalak has studied reconstruction of Morse functions on closed manifolds whose general contours are spheres, in \cite{michalak}. Later, the author has contributed to this, respecting not only graphs, but also topologies of contours. Related important studies are \cite{kitazawa1, kitazawa4}. In addition, \cite{kitazawa2} is a study on the {\it non-compact} manifold case and the so-called {\it non-proper} case. \cite{kitazawa3} is a pioneering study on reconstruction in the real algebraic situation and a method used there is also important in constructing the map into ${\mathbb{R}}^2$ in Problem \ref{prob:1}. This real algebraic case is followed by the author himself in various preprints such as \cite{kitazawa5, kitazawa6}, where we do not assume related knowledge. 

Recently, the author has launched studies on Reeb spaces which are not finite graphs. He has been interested in explicit real analytic functions or densely real analytic ones. \cite{kitazawa7, kitazawa8} are preprints on pioneering cases, where we discuss related facts and arguments in a self-contained way in the present paper without assuming related knowledge. In Problem \ref{prob:1}, we respect and generalize the situation of \cite[Theorem 1]{kitazawa7}.

This paper is a note and a kind of remarks on Problem \ref{prop:1} and \cite{kitazawa9, kitazawa10}. We investigate local structures of Reeb spaces $R_c$ of the functions $c:={\pi}_{m+1,1} {\mid}_{X_{m,c_1,c_2}}$ and global structures of the functions roughly (Theorems \ref{thm:1}, \ref{thm:2} and \ref{thm:3}). It is important that we have Reeb spaces $R_c$ of the function of Problem \ref{prob:1} which are not of types we have discussed previously. We encounter Reeb spaces which are not CW complexes, for example. We also present new cases which are not presented in \cite{kitazawa9, kitazawa10} and which are examples related to the present paper (Theorem \ref{thm:4}). The 2nd section is for preliminaries and fundamental notions on singularity of differentiable maps, graphs and (CW) complexes and graph structures of Reeb spaces, are reviewed. In the 3rd section, our main result such as Theorems \ref{thm:1}, \ref{thm:2}, \ref{thm:3}, and \ref{thm:4} is presented.

\section{Preliminaries.}

%Let $\mathbb{Z} \subset \mathbb{R}$ denote the set of all integers and $\mathbb{N}$ denote the set of all positive integers.
A map between topological spaces $c:X \rightarrow Y$ is {\it proper} if the preimage $c^{-1}(K)$ is compact for any compact subset $K$ of $Y$. A {\it non-proper} map is a map between topological spaces which is not proper.
\subsection{Differentiable manifolds and maps.}
For a differentiable manifold $X$, the tangent vector space of at $p \in X$ is denoted by $T_p X$. This is a real vector space of dimension same as that of $X$. For a differentiable map $c:X_1 \rightarrow X_2$ between differentiable manifolds, a {\it singular} point of $c$ means a point $p \in X$ where the rank of the differential ${dc}_p:T_p X_1 \rightarrow T_{c(p)} X_2$, which is a linear map, is smaller than both the dimension of $X_1$ and that of $X_2$. Let $S(c)$ denote the set of all singular points of $c$ and we call it the {\it singular set} of $c$. 
We call the image $c(S(c))$ the {\it singular value set} of $c$. %We call $c(p)$ ($p \in S(c)$) a {\it singular value} of $c$.
If $X_2$ is of $1$-dimensional, then, we use "{\it critical}" instead of "singular".

Let
$D^n:=\{x \in {\mathbb{R}}^n \mid {\Sigma}_{j=1}^n {x_j}^2 \leq 1\}$ denote the $n$-dimensional unit disk.
Let $S^{n-1}$ denote its boundary $\{x \in {\mathbb{R}}^n \mid {\Sigma}_{j=1}^n {x_j}^2 = 1\}$, which is the ($n-1$)-dimensional unit sphere. Let $D^0:=\{0\}$ denote a one-point set. We also use $S^{-1}$ for the empty set $\emptyset$.

\subsection{(CW) complexes and graphs.}

A {\it cell complex} $\{e_{j,{\lambda}_j}:D^j-S^{j-1} \rightarrow X\}_{{\lambda}_j \in \Lambda}$ means a family of maps into a Hausdorff space $X$ satisfying the following. We call each map $e_{j,{\lambda}_j}$ a {\it cell} of it. We also call its image a {\it cell}. We call $X$ the {\it underlying space} of the cell complex. 
\begin{itemize}
\item The images of distinct maps $e_{j,{\lambda}_j}$ are mutually disjoint. The disjoint union of the images of all maps $e_{j,{\lambda}_j}$ is $X$.
\item The restriction of each map $e_{j,{\lambda}_j}$ to $D^j-S^{j-1}$ is a homeomorphism onto a subspace of $X$. 
\item The closure of the image of each map $e_{j,{\lambda}_j}$ taken in $X$ is the disjoint union of the image and the images of several maps $e_{j^{\prime},{\lambda}_j}$ with $j^{\prime}<j$.  
\end{itemize}
Each map $e_{j,{\lambda}_j}$ and its image is a {\it $j$-cell}.
A complex is {\it locally finite} if each point $x \in X$ is contained in images of finitely many maps $e_{j,{\lambda}_j}$.
We can consider a {\it finite complex} as a cell complex the number of all maps $e_{j,{\lambda}_j}$ into $X$ is finite.
A {\it subcomplex} of a cell complex is defined naturally by considering a suitable subspace $Y \subset X$.
A CW {\it complex} is a cell complex satisfying the following.
\begin{itemize}
\item In the underlying space $X$, the closure of each cell is a disjoint union of finitely many cells.  
\item A subset $O$ of $X$ is open in $X$ if and only if the set $O \bigcap e_{j,{\lambda}_j}(D^j-S^{j-1})$ is open in the subspace $e_{j,{\lambda}_j}(D^j-S^{j-1}) \subset X$.
\end{itemize}
For a CW complex, we can define its dimension as the maximal $j$ for $e_{j,{\lambda}_j}$ or $+\infty$. This is a topological invariant for topological spaces homeomorphic to a CW complex. Smooth manifolds and so-called PL manifolds are of such a space. For related theory, see \cite{hatcher} for example.

A {\it graph} means a $1$-dimensional locally finite CW complex such that the closure of each $1$-cell is homeomorphic to $D^1$. A {\it graph with ends} or an {\it E-graph} means a $1$-dimensional locally finite CW complex such that the closure of each cell is homeomorphic to $D^1$, $\{x>0\}$, or $\mathbb{R}$.

An {\it  almost graph} ({\it with ends}) or an {\it A-graph} (resp. {\it E-A-graph}) means a $1$-dimensional CW complex such that the closure of each $1$-cell is homeomorphic to $D^1$ (resp. $D^1$, $\{x>0\}$, or $\mathbb{R}$). 

A {\it  weakly almost graph} ({\it with ends}) or a {\it W-A-graph} (resp. an {\it E-W-A-graph}) means a $1$-dimensional cell complex such that the closure of each $1$-cell is homeomorphic to $D^1$ (resp. $D^1$, $\{x>0\}$, or $\mathbb{R}$). This is the widest among the presented classes of cell complexes.

For an E-W-A-graph, we use a single character $G$ for example and a $1$-cell ($0$-cell) is also called an {\it edge} (resp. a {\it vertex}). 

We can consider an oriented version if there exists a continuous real-valued function $c_G:G \rightarrow \mathbb{R}$ on an E-W-A-graph $G$ whose restriction to each $1$-cell is injective. More rigorously, an edge $e_v$ incident to a vertex $v$ is oriented as an edge departing from (entering) $v$ if the restriction of $c_G$ to the closure of $e_v$ in the graph has the minimum (resp. maximum) at $v$. For the pair $(G,c_G)$, we use "{\it digraph}" instead of "graph", in these terminologies.
\subsection{Reeb spaces and Reeb (di)graphs.}
For a non-wild proper smooth function $c:X \rightarrow \mathbb{R}$ on a manifold $X$ with no boundary such that $c(S(c))$ is discrete and closed in $\mathbb{R}$, $R_c$ is a graph (with ends) whose vertex set consists of all points for critical contours of $c$ and the {\it Reeb graph} of $c$ (\cite[Theorem 3.1]{saeki1}). Note that Reeb graphs are classical tools and have been strong tools in understanding manifolds by nice smooth functions, since pioneering studies such as \cite{reeb}. 
For the Reeb graph $G:=R_c$ with $c_G:=\bar{c}$, the pair $(R_c,\bar{c})$ is a digraph. This is the {\it Reeb digraph} of $c$. In this paper, we consider Reeb spaces which are not homeomorphic to any E-graph or any E-A-graph and are E-W-A-(di)graphs by this rule, first. We also call such a graph the {\it Reeb} ({\it di}){\it graph} of $c$. 
\section{Our main result.}
First, we review a proof of the fact that $X_{m,c_1,c_2}$ in Problem \ref{prob:1} is a smooth submanifold with no boundary of ${\mathbb{R}}^{m+1}$ again, as we have done in \cite{kitazawa9, kitazawa10}. Its main ingredient is essentially in \cite{kitazawa3} and this explicit case respects the case of \cite[Theorem 1]{kitazawa7}.

We can easily see that this is the zero set of the function $(x_1-c_1(x_2))(c_2(x_2)-x_1)-{\Sigma}_{j=1}^{m-1} {y_j}^2$. For this, we use implicit function theorem. More precisely, we calculated the rank of the differential of the function at each point of the zero set. If $(x_1,x_2) \in D_{c_1,c_2}$ and $(x_1,x_2,y) \in X_{m,c_1,c_2}$, then the value of the derivative of the function by some $y_j$ is not zero. If $(x_1,x_2) \in \overline{D_{c_1,c_2}}-D_{c_1,c_2}$ and $(x_1,x_2,y) \in X_{m,c_1,c_2}$, then the value of the derivative of the function by $x_1$ is not zero since either the relation $x_1-c_1(x_2)=0$ or $c_2(x_2)-x_1=0$ holds at the point. This means that the rank is $1$. We can apply implicit function theorem to see that  $X_{m,c_1,c_2}$ is an $m$-dimensional smooth submanifold with no boundary of ${\mathbb{R}}^{m+1}$.

We restrict classes of cases of Problem \ref{prob:1}.

\begin{Def}
\label{def:1}

\begin{enumerate}
\item In Problem \ref{prob:1}, let each $S(c_i)$ be the disjoint union of connected sets in $\mathbb{R}$ diffeomorphic to a one-point set, $D^1$, $\{x>0\}$, or $\mathbb{R}$: note that a connected component of $S(c_i)$ is diffeomorphic to $\mathbb{R}$, if and only if $c_i$ is constant and satisfies $S(c_i)=\mathbb{R}$. In this case,  the quadruple $(c_1,c_2,D_{c_1,c_2},X_{m,c_1,c_2})$ is said to be {\it tame from the viewpoint of singularity theory} or {\it TS}.
\item  If $(c_1,c_2,D_{c_1,c_2},X_{m,c_1,c_2})$ is TS and $S(c_i)$ is discrete, then this is said to be discrete-closed or {\it DC}.
\item If $(c_1,c_2,D_{c_1,c_2},X_{m,c_1,c_2})$ is TS (DC) and $c_1(S(c_1) \bigcup c_2(S(c_2))$ is a discrete and closed set of the complement $\mathbb{R}-Z_{\rm F}$ of some discrete closed set $Z_{\rm F}$ of $\mathbb{R}$, then it is {\it normally TS} (resp .{\it normally DC}) or {\it NTS} (resp. {\it NDC}).
\item If $(c_1,c_2,D_{c_1,c_2},X_{m,c_1,c_2})$ is NTS (NDC) and the set $Z_{\rm F}$ can be chosen to be empty, then it is {\it purely TS} (resp. {\it purely DC}) or {\it PTS} (resp. {\it PDC}).

\end{enumerate}
\end{Def}
We have Proposition \ref{prop:1}. Its proof is a kind of fundamental exercises. This is also in \cite{kitazawa9, kitazawa10}.
\begin{Prop}
\label{prop:1}
\begin{enumerate}
\item \label{prop:1.1} A point $p \in X_{m,c_1,c_2}$ is a critical point of the function ${\pi}_{m+1,1} {\mid}_{X_{m,c_1,c_2}}$ if and only if $p$ is of the form $p=(c_i(x)),x,{(0)}_{j=1}^{m-1})$ {\rm (}$x \in S(c_i)${\rm )}.
\item \label{prop:1.2} If two functions $c_i$ are Morse {\rm (}real analytic{\rm )}, then $(c_1,c_2,D_{c_1,c_2},X_{m,c_1,c_2})$ is {\it DC} and the function ${\pi}_{m+1,1} {\mid}_{X_{m,c_1,c_2}}$ is Morse {\rm (}resp. the restriction of the zero set of the presented real analytic function $(x_1-c_1(x_2))(c_2(x_2)-x_1)-{\Sigma}_{j=1}^{m-1} {y_j}^2${\rm )}.

Related to this, we call a contour of ${\pi}_{2,1} {\mid}_{\overline{D_{c_1,c_2}}}$ containing such a point a {\rm critical} contour of ${\pi}_{2,1} {\mid}_{\overline{D_{c_1,c_2}}}$.
\end{enumerate}
\end{Prop}
Theorem \ref{thm:1} is partially studied in \cite{kitazawa10}. The resulting Reeb graphs are $1$-dimensional CW complexes there. In short, Case 1-2, presented in its proof, later, does not appear, previously.

\begin{Thm}
\label{thm:1}
For an NTS quadruple $(c_1,c_2,D_{c_1,c_2},X_{m,c_1,c_2})$, the Reeb space $R_{{\pi}_{m+1,1} {\mid}_{X_{m,c_1,c_2}}}$ is homeomorphic to an E-W-A-graph. 
The function $\bar{{\pi}_{m+1,1} {\mid}_{X_{m,c_1,c_2}}}$ is represented as the composition of an $e_{R_{{\pi}_{m+1,1} {\mid}_{X_{m,c_1,c_2}}}}:R_{{\pi}_{m+1,1} {\mid}_{X_{m,c_1,c_2}}} \rightarrow {\mathbb{R}}^2$ with ${\pi}_{2,1}$.
The pair $(R_{{\pi}_{m+1,1} {\mid}_{X_{m,c_1,c_2}}},\bar{{\pi}_{m+1,1} {\mid}_{X_{m,c_1,c_2}}})$ is the Reeb digraph which is an E-A-digraph if $(c_1,c_2,D_{c_1,c_2},X_{m,c_1,c_2})$ is PTS. 

.
\end{Thm}
\begin{proof}

Each contour $C_p \subset {{\pi}_{2,1}}^{-1}(p) \bigcap \overline{D_{c_1,c_2}}$ of ${\pi}_{2,1} {\mid}_{\overline{D_{c_1,c_2}}}$ is homeomorphic to a one-point set, $D^1$, $\mathbb{R}$, or $\{x>0\}$ and ${{\pi}_{2,1}}^{-1}(p) \bigcap \overline{D_{c_1,c_2}}$ is the disjoint union of such sets and closed in ${{\pi}_{2,1}}^{-1}(p)$, which is diffeomorphic to $\mathbb{R}$. 
Contours $F_{C_p} \subset {{\pi}_{m+1,1}}^{-1}(p) \bigcap \overline{X_{m,c_1,c_2}}$ of ${\pi}_{m+1,1} {\mid}_{X_{m,c_1,c_2}}$ and the contours $C_p$ of ${\pi}_{2,1} {\mid}_{\overline{D_{c_1,c_2}}}$ correspond one-to-one.
For each contour $C_p$ of ${\pi}_{2,1} {\mid}_{\overline{D_{c_1,c_2}}}$, we can argue as follows.
\begin{itemize}
\item We can have the unique connected component $K(C_p)$ containing $C_p$ in a closed set $\overline{D_{c_1,c_2}} \bigcap \{p_{C_p,{\rm l}} \leq x \leq p_{C_p,{\rm L}}\}$ of the closed set  $\overline{D_{c_1,c_2}}$ with two suitable real numbers $p_{C_p,{\rm l}}$ and $p_{C_p,{\rm L}}$ such that $p_{C_p,{\rm l}}<p<p_{C_p,{\rm L}}$ and that the difference $p_{C_p,{\rm L}}-p_{C_p,{\rm l}}$ is sufficiently small. Of course $\overline{D_{c_1,c_2}}$ is a closed set of ${\mathbb{R}}^2$.

\item $K(C_p)$ does not contain any other contour ${C_p}^{\prime} \subset {{\pi}_{2,1}}^{-1}(p) \bigcap \overline{D_{c_1,c_2}}$ of the function ${\pi}_{2,1} {\mid}_{\overline{D_{c_1,c_2}}}$.
\end{itemize}

By considering distinct points $p \in {\pi}_{2,1}(\overline{D_{c_1,c_2}})$, we can see that distinct contours $C_p$ of the function ${\pi}_{2,1} {\mid}_{\overline{D_{c_1,c_2}}}$ are separated by connected and closed neighborhoods
in $\overline{D_{c_1,c_2}}$ and this guarantees that the spaces $R_{{\pi}_{2,1} {\mid}_{\overline{D_{c_1,c_2}}}}$ and $R_{{\pi}_{m+1,1} {\mid}_{X_{m,c_1,c_2}}}$ are Hausdorff. Note also that by our situations, these two Reeb spaces are homeomorphic and investigating the structure of an E-W-A-digraph of $R_{{\pi}_{2,1} {\mid}_{\overline{D_{c_1,c_2}}}}$ is essentially same as that of $R_{{\pi}_{m+1,1} {\mid}_{X_{m,c_1,c_2}}}$.

We consider two cases for choice of closed sets $\overline{D_{c_1,c_2}} \bigcap \{x \in \mid p_{C_p,{\rm l}} \leq x \leq p_{C_p,{\rm L}} \}$ and its connected component $K(C_p)$. Due to our definition and situation, $K(C_p)-C_p$ is a union of contours of ${\pi}_{2,1} {\mid}_{\overline{D_{c_1,c_2}}}$\\
 \ \\
Case 1-1 We can have $K(C_p)$ in such a way that $K(C_p)-C_p$ contains no critical contour of ${\pi}_{2,1} {\mid}_{\overline{D_{c_1,c_2}}}$.\\
For each connected component of $K(C_p)-C_p$, its closure in $K(C_p)$ must contain some point of $C_p$. If it does not contain one, then due to our definition and situation, the connected component must be closed and open in $K(C_p)$ and a connected component of $K(C_p)-C_p$ and this is a contradiction. By our definition and situation, each connected component of $K(C_p)-C_p$ is mapped onto a $1$-dimensional connected manifold which is non-compact and of the form $\{0 \leq x<1\}$ by the quotient map $q_{{\pi}_{2,1} {\mid}_{\overline{D_{c_1,c_2}}}}$. $K(C_p)$ is, by the quotient map, mapped onto a space obtained by identifying $1 \in \{0 \leq x \leq 1\}$ of all given copies of $\{0 \leq x \leq1\}$. A contour of ${\pi}_{2,1} {\mid}_{\overline{D_{c_1,c_2}}}$ which is a subset of $K(C_p)$ is mapped to the point identified with $1 \in \{0 \leq x \leq1\}$ if and only if it is a critical contour of ${\pi}_{2,1} {\mid}_{\overline{D_{c_1,c_2}}}$ and contained in $K(C_p)$. \\
 \ \\
Case 1-2 The case we cannot do as in Case 1-1. \\
Note that $p \in Z_{\rm F}$ must hold and that $C_p$ is not homeomorphic to a one-point set or $D^1$, for Case 1-2. In the PTS case, we cannot encounter Case 1-2. 
For each connected component of $K(C_p)-C_p$, its closure in $K(C_p)$ must contain some point of $C_p$ and this is due to the corresponding argument in Case 1-1.  \\
\ \\
Note that except sets $C_p$ of Case 1-2, which are named {\it $(c_1,c_2,D_{c_1,c_2},X_{m,c_1,c_2})$-non-normal sets}, in Theorem \ref{thm:2}, later, each contour of ${\pi}_{2,1} {\mid}_{\overline{D_{c_1,c_2}}}$ is mapped to a point in a local $1$-dimensional cell (CW) complex as in Case 1-1. According to the argument of Case 1-1, by removing all points $p_{C_p}$ some $C_p$ of Case 1-2 mapped to by the quotient map $q_{{\pi}_{2,1} {\mid}_{\overline{D_{c_1,c_2}}}}$, we have a $1$-dimensional CW-complex. By the argument, the Reeb space is homeomorphic to the natural $1$-dimensional cell complex, where the set of all points $p_{C_p}$ (, in Case 1-2,) mapped to by the quotient map are discrete in $R_{{\pi}_{m+1,1} {\mid}_{X_{m,c_1,c_2}}}$. In addition, these points $p_{C_p}$ can be defined as vertices of the E-W-A-graph $R_{{\pi}_{m+1,1} {\mid}_{X_{m,c_1,c_2}}}$ (which may not yield a Reeb-digraph). 

Case 1-2 also implies that the pair $(R_{{\pi}_{m+1,1} {\mid}_{X_{m,c_1,c_2}}},\bar{{\pi}_{m+1,1} {\mid}_{X_{m,c_1,c_2}}})$ may be an E-W-A-digraph which may not be an E-A-digraph, or that we cannot have the Reeb-digraph $(R_{{\pi}_{m+1,1} {\mid}_{X_{m,c_1,c_2}}},\bar{{\pi}_{m+1,1} {\mid}_{X_{m,c_1,c_2}}})$ with $R_{{\pi}_{m+1,1} {\mid}_{X_{m,c_1,c_2}}}$ being homeomorphic to an E-W-A-graph. In other words, any neighborhood of the uniquely defined point $p_{C_p}$ may contain other vertices of the digraph $(R_{{\pi}_{m+1,1} {\mid}_{X_{m,c_1,c_2}}},\bar{{\pi}_{m+1,1} {\mid}_{X_{m,c_1,c_2}}})$ in the former case and $R_{{\pi}_{m+1,1} {\mid}_{X_{m,c_1,c_2}}}$ may not have the structure of an E-W-A-digraph by the rule for vertices of Reeb (di)graphs in the latter case. 

Our desired embedding $e_{R_{{\pi}_{m+1,1} {\mid}_{X_{m,c_1,c_2}}}}:R_{{\pi}_{m+1,1} {\mid}_{X_{m,c_1,c_2}}} \rightarrow {\mathbb{R}}^2$ can be obtained by choosing corresponding points of the form $(c_i(x),x)$ ($x \in S(c_i)$) and connected curves satisfying the following.
\begin{itemize}
\item Each curve is the image of the embedding of an (oriented) edge of the graph $R_{{\pi}_{2,1} {\mid}_{\overline{D_{c_1,c_2}}}}$ and oriented canonically.
\item The image of the embedding is a subspace of $\overline{D_{c_1,c_2}}$.
\item Either the following holds for each (oriented) curve $C_e$.
\begin{itemize}
\item  $C_e$ is homeomorphic to $\{x>0\}$ and departs from a chosen point from these important points $(c_i(x),x)$ ($x \in S(c_i)$) and contains no such a point other than this.
\item $C_e$ is homeomorphic to $\{x<0\}$  and enters a chosen point from these important points $(c_i(x),x)$ ($x \in S(c_i)$) and contains no such a point other than this.
\item $C_e$ is homeomorphic to $D^1$ and connecting two suitably chosen points $(c_i(x),x)$ ($x \in S(c_i)$) from these important points and contains no such a point other than these two.
\item $C_e$ is homeomorphic to $\mathbb{R}$ and contains no point of the form $(c_i(x),x)$ ($x \in S(c_i)$). This occurs if and only if $S(c_1)$ and $S(c_2)$ are both empty.
\end{itemize}
\end{itemize}
This completes the proof.
\end{proof}
\begin{Thm}
\label{thm:2}	
In Theorem \ref{thm:1}, the following hold. Hereafter, we say that a contour $C_p$ of  $\bar{{\pi}_{2,1} {\mid}_{\overline{D_{c_1,c_2}}}}$ as in Case 1-2 of the proof of Theorem \ref{thm:1} and the naturally corresponding contour $F_{C_p}$ of ${\pi}_{m+1,1} {\mid}_{X_{m,c_1,c_2}}$ are {\rm $(c_1,c_2,D_{c_1,c_2},X_{m,c_1,c_2})$-non-normal}.
\begin{enumerate}
\item \label{thm:2.1} The pair $(R_{{\pi}_{m+1,1} {\mid}_{X_{m,c_1,c_2}}},\bar{{\pi}_{m+1,1} {\mid}_{X_{m,c_1,c_2}}})$ is the Reeb digraph and an E-W-A-digraph if and only if all $(c_1,c_2,D_{c_1,c_2},X_{m,c_1,c_2})$-non-normal sets $C_p$ are critical.
\item \label{thm:2.2}
The Reeb digraph of {\rm (}\ref{thm:2.1}{\rm )}, presented above, is not an E-A-digraph if and only if there exists some $(c_1,c_2,D_{c_1,c_2},X_{m,c_1,c_2})$-non-normal contour $C_p$ of  $\bar{{\pi}_{2,1} {\mid}_{\overline{D_{c_1,c_2}}}}$.
% In addition, the E-W-A-digraph which is not an E-A-digraph is homeomorphic to an E-A-graph if and only if for each $(c_1,c_2,D_{c_1,c_2},X_{m,c_1,c_2})$-non-normal set $C_p$ satisfies the following.
 \end{enumerate}

\end{Thm}
\begin{proof}
We prove (\ref{thm:2.1}).

Suppose that all $(c_1,c_2,D_{c_1,c_2},X_{m,c_1,c_2})$-non-normal sets $C_p$ are critical. Under the assumption, they are mapped to single points $p_{C_p}$ of $R_{{\pi}_{2,1} {\mid}_{\overline{D_{c_1,c_2}}}}$ by the quotient map $q_{{\pi}_{2,1} {\mid}_{\overline{D_{c_1,c_2}}}}$. By the proof of Theorem \ref{thm:1}, the pair $(R_{{\pi}_{m+1,1} {\mid}_{X_{m,c_1,c_2}}},\bar{{\pi}_{m+1,1} {\mid}_{X_{m,c_1,c_2}}})$ is the Reeb digraph, where these single points are vertices of the Reeb digraph. 

On the contrary, suppose that there exists a $(c_1,c_2,D_{c_1,c_2},X_{m,c_1,c_2})$-non-normal set $C_p$ which is not critical. Suppose also that the pair $(R_{{\pi}_{m+1,1} {\mid}_{X_{m,c_1,c_2}}},\bar{{\pi}_{m+1,1} {\mid}_{X_{m,c_1,c_2}}})$ is the Reeb digraph. Under this assumption, this is mapped to a point $p_{C_p}$ which is not a vertex of the Reeb digraph which is an E-W-A-digraph. We find a contradiction.

We introduce two suitably chosen sufficiently small positive numbers ${\epsilon}_i$ ($i=1,2$).
Any small open and connected set in $\overline{D_{c_1,c_2}}$ containing $C_p$ must be mapped to a set of the form $(p-{\epsilon}_1,p+{\epsilon}_2):=\{p-{\epsilon}_1<x<p+{\epsilon}_2\}$ by ${\pi}_{2,1} {\mid}_{\overline{D_{c_1,c_2}}}$ if ${\pi}_{2,1} {\mid}_{\overline{D_{c_1,c_2}}}$ does not have local extremum on $C_p$. Any small open and connected set in $\overline{D_{c_1,c_2}}$ containing $C_p$ must be mapped to a set of either the form $(p-{\epsilon}_1,p):=\{p-{\epsilon}_1<x<p\}$ or  $(p,p+{\epsilon}_2):=\{p<x<p+{\epsilon}_2\}$ by ${\pi}_{2,1} {\mid}_{\overline{D_{c_1,c_2}}}$ if ${\pi}_{2,1} {\mid}_{\overline{D_{c_1,c_2}}}$ has a local extremum on $C_p$.

From this, the preimage of each neighborhood of the one-point set $\{p_{C_p}\}$ by the quotient map $q_{{\pi}_{2,1} {\mid}_{\overline{D_{c_1,c_2}}}}$ must contain countably many points of the form $(c_i(x)),x)$ ($x \in S(c_i)$). We introduce two suitably chosen sufficiently small positive numbers ${\epsilon}_i$ ($i=1,2$).
Any small open and connected set in $\overline{D_{c_1,c_2}}$ containing $C_p$ must be mapped to a set of the form $(p-{\epsilon}_1,p+{\epsilon}_2):=\{p-{\epsilon}_1<x<p+{\epsilon}_2\}$ by ${\pi}_{2,1} {\mid}_{\overline{D_{c_1,c_2}}}$ if ${\pi}_{2,1} {\mid}_{\overline{D_{c_1,c_2}}}$ does not have local extremum on $C_p$. Any small open and connected set in $\overline{D_{c_1,c_2}}$ containing $C_p$ must be mapped to a set of either the form $(p-{\epsilon}_1,p):=\{p-{\epsilon}_1<x<p\}$ or  $(p,p+{\epsilon}_2):=\{p<x<p+{\epsilon}_2\}$ by ${\pi}_{2,1} {\mid}_{\overline{D_{c_1,c_2}}}$ if ${\pi}_{2,1} {\mid}_{\overline{D_{c_1,c_2}}}$ has a local extremum on $C_p$.

From this, the preimage of each neighborhood of the one-point set $\{p_{C_p}\}$ by the quotient map $q_{{\pi}_{2,1} {\mid}_{\overline{D_{c_1,c_2}}}}$ must contain countably many points of the form $(c_i(x)),x)$ ($x \in S(c_i)$). %Each open connected neighborhood of $C_p$ in $\overline{D_{c_1,c_2}}$ must be mapped to a connected set $C_{p_{C_p}}$ which contains the one-point set $\{p_{C_p}\}$ as a proper subset, and the preimage ${q_{{\pi}_{2,1} {\mid}_{\overline{D_{c_1,c_2}}}}}^{-1}(C_{p_{C_p}})$ of the image $C_{p_{C_p}}$ must contain $C_p$ and at least one critical contour of ${\pi}_{2,1} {\mid}_{\overline{D_{c_1,c_2}}}$ which is not $C_p$ by the proof of Theorem \ref{thm:1} and our definition. 
This means that any neighborhood of $p_{C_p}$ in the E-W-A-graph must contain vertices of the E-W-A-graph and this is a contradiction.

This completes the proof of Theorem \ref{thm:2} (\ref{thm:2.1}).

We prove (\ref{thm:2.2}).

The Reeb digraph is an E-A-digraph if we do not have any $(c_1,c_2,D_{c_1,c_2},X_{m,c_1,c_2})$-non-normal $C_p$. Assume that we have some $(c_1,c_2,D_{c_1,c_2},X_{m,c_1,c_2})$-non-normal $C_p$. In this case, first, the following is satisfied.
\begin{itemize}
\item $C_p={{\pi}_{2,1}}^{-1}(p)$ and diffeomorphic to $\mathbb{R}$
\item $C_p=\{(p,x_2) \mid x_2 \leq p_{\rm m}\} \subset K(C_p)$ for some $p_{\rm m} \in \mathbb{R}$. 
\item $C_p=\{(p,x_2) \mid x_2 \geq p_{\rm M}\} \subset K(C_p)$ for some $p_{\rm M} \in \mathbb{R}$. 
\end{itemize}

Hereafter, we define ${\pi}_{2,2}:{\mathbb{R}}^2 \rightarrow \mathbb{R}$ as the projection onto the second component (${\pi}_{2,2}(x_1,x_2):=x_2$). In short, ${\pi}_{2,2}(C_p)$ is not bounded.

For $C_p$ and $i=1,2$, suppose the following. We do this to find a contradiction.
\begin{itemize}
\item Under the constraint that ${\pi}_{2,2}(C_p)$ is unbounded above,  ${c_i}^{-1}(p) \bigcap {\pi}_{2,2}(C_p)$ is unbounded above.
\item Under the constraint that ${\pi}_{2,2}(C_p)$ is unbounded below, ${c_i}^{-1}(p) \bigcap {\pi}_{2,2}(C_p)$ is unbounded below.
\end{itemize}
Under this assumption, the closure of each connected component of ${\pi}_{2,2}(C_p) \bigcap (\mathbb{R}-{c_i}^{-1}(p))$ in $\mathbb{R}$ must be always homeomorphic to $D^1$ and a bounded and closed subset of $\mathbb{R}$. By the assumption that $S(c_i) \subset \mathbb{R}$ is closed in $\mathbb{R}$ and the representation of $S(c_i)$, we can do as in Case 1-1 of the proof of Theorem \ref{thm:1}. This is a contradiction.

From this argument, at least one of the following four must hold.

\begin{itemize}
\item ${\pi}_{2,2}(C_p)$ is unbounded above. ${c_1}^{-1}(p) \bigcap {\pi}_{2,2}(C_p)$ is bounded above. There exists a sequence $\{x_{c_1,j}\}_{j=1}^{\infty} \subset S(c_1)$ diverging to $\infty$ and making $\{c_1(x_{c_1,j})\}$ converging to $p$.
\item ${\pi}_{2,2}(C_p)$ is unbounded above. ${c_2}^{-1}(p) \bigcap {\pi}_{2,2}(C_p)$ is bounded above. There exists a sequence $\{x_{c_2,j}\}_{j=1}^{\infty} \subset S(c_2)$ diverging to $\infty$ and making $\{c_2(x_{c_2,j})\}$ converging to $p$.
\item ${\pi}_{2,2}(C_p)$ is unbounded below. ${c_1}^{-1}(p) \bigcap {\pi}_{2,2}(C_p)$ is bounded below. There exists a sequence $\{x_{c_1,j}\}_{j=1}^{\infty} \subset S(c_1)$ diverging to $-\infty$ and making $\{c_1(x_{c_1,j})\}$ converging to $p$.
\item ${\pi}_{2,2}(C_p)$ is unbounded below. ${c_2}^{-1}(p) \bigcap {\pi}_{2,2}(C_p)$ is bounded below. There exists a sequence $\{x_{c_2,j}\}_{j=1}^{\infty} \subset S(c_2)$ diverging to $-\infty$ and making $\{c_2(x_{c_2,j})\}$ converging to $p$.
\end{itemize}

As in the proof of (\ref{thm:2.1}), we can prove that any neighborhood $C_{p_{C_p}}$ of the one-point set $\{p_{C_p}\}$ in the Reeb digraph, the set $C_p$ is mapped to by the quotient map, must contain vertices of the Reeb digraph (which is an E-W-A-digraph). This E-W-A-digraph is not an E-A-digraph.

This completes the proof of Theorem \ref{thm:2} (\ref{thm:2.2}).

This completes the proof.
\end{proof}
\begin{Def}
In the end of the proof of Theorem \ref{thm:2}, we call each sequence $\{x_{c_i,j}\}_{j=1}^{\infty} \subset S(c_i)$ a {\it characteristic sequence} at a $(c_1,c_2,D_{c_1,c_2},X_{m,c_1,c_2})$-non-normal set $C_p$. In addition, if for any characteristic sequence $\{x_{c_i,j}\}_{j=1}^{\infty} \subset S(c_i)$ at $C_p$, we can choose some connected set $I_{\{x_{c_i,j}\}_{j=1}^{\infty}}$ which is homeomorphic to $\{x>0\}$ and $\{x<0\}$, which contains some subsequence of $\{x_{c_i,j}\}_{j=1}^{\infty}$ and on which $c_i$ is strongly increasing or strongly decreasing, then we say that the $(c_1,c_2,D_{c_1,c_2},X_{m,c_1,c_2})$-non-normal set $C_p$ and the corresponding set $F_{C_p}$ are {\it mild}.  
\end{Def}
\begin{Thm}
\label{thm:3}	
In Theorems \ref{thm:1} and \ref{thm:2}, for the case where the pair $(R_{{\pi}_{m+1,1} {\mid}_{X_{m,c_1,c_2}}},\bar{{\pi}_{m+1,1} {\mid}_{X_{m,c_1,c_2}}})$ does not give the Reeb digraph, or gives the Reeb digraph which is not an E-A-digraph, the Reeb space $R_{{\pi}_{m+1,1} {\mid}_{X_{m,c_1,c_2}}}$ is homeomorphic to an E-A-graph if and only if each $(c_1,c_2,D_{c_1,c_2},X_{m,c_1,c_2})$-non-normal set $C_p$ and the corresponding contour $F_{C_p}$ of ${\pi}_{m+1,1} {\mid}_{X_{m,c_1,c_2}}$ are mild.

\end{Thm}
\begin{proof}
By the proof of Theorems \ref{thm:1} and \ref{thm:2} and discussions on $(c_1,c_2,D_{c_1,c_2},X_{m,c_1,c_2})$-non-normal sets $C_p$, we can choose a connected open neighborhood $C_{p_{C_p}}$ in the Reeb space $R_{{\pi}_{m+1,1} {\mid}_{X_{m,c_1,c_2}}}$ homeomorphic to $\mathbb{R}$ and $\{-1<x<1\}$ of the one-point-set $\{p_{C_p}\}$, to which $C_p$ is mapped by the quotient map $q_{{\pi}_{2,1} {\mid}_{\overline{D_{c_1,c_2}}}}$, if $C_p$ and $F_{C_p}$ are mild. We can also see that if some $C_p$ ($F_{C_p}$) is not mild, then such a set makes $R_{{\pi}_{m+1,1} {\mid}_{X_{m,c_1,c_2}}}$ and $R_{{\pi}_{2,1} {\mid}_{\overline{D_{c_1,c_2}}}}$ spaces which are not homeomorphic to any E-A-graphs. This, with arguments on our proof, especially the argument of Case 1-1 in the proof of Theorem \ref{thm:1}, completes the proof.
\end{proof}
Corollary \ref{cor:1} is due to the definition.
\begin{Cor}
\label{cor:1}
Suppose the following.
\begin{itemize}
\item There exists a real number $x_{\rm m}$. The set $c_i(S(c_i))\bigcap \{x \leq x_{\rm m}\}$ is discrete and closed in either $\mathbb{R}$ or the complementary set of some suitably chosen one-point set $Z_{{\rm m},i}$ in $\mathbb{R}$.
\item There exists a real number $x_{\rm M}$. The set $c_i(S(c_i)) \bigcap \{x \geq x_{\rm M}\}$ is discrete and closed in either $\mathbb{R}$ or the complementary set of some suitably chosen one-point set $Z_{{\rm M},i}$ in $\mathbb{R}$.
\end{itemize}
 Then, a TS quadruple $(c_1,c_2,D_{c_1,c_2},X_{m,c_1,c_2})$ is NTS and the finite set $Z_{\rm F}$ in Definition \ref{def:1} can be chosen to be ${\bigcup}_{i=1}^2 (Z_{{\rm m},i} \bigcup Z_{{\rm M},i})$ and a finite discrete set of at most four elements, where each set $Z_{{\rm m},i}$ {\rm (}resp. $Z_{{\rm M},i}${\rm )} is empty if $c_i(S(c_i))\bigcap \{x \leq x_{\rm m}\}$ {\rm (}resp. $c_i(S(c_i))\bigcap \{x \geq x_{\rm M}\}${\rm )} is discrete and closed in $\mathbb{R}$.

We say that the present case $(c_1,c_2,D_{c_1,c_2},X_{m,c_1,c_2})$ is {\rm minimal} or {\rm MNTS} {\rm (}{\rm MNDC in the case NDC}{\rm )}.
\end{Cor}
We present examples for this.
\begin{Ex}
	\label{ex:1}
\begin{enumerate}
\item \label{ex:1.1} A PTS quadruple $(c_1,c_2,D_{c_1,c_2},X_{m,c_1,c_2})$ is also minimal. 
\item \label{ex:1.2} In \cite{kitazawa9}, certain explicit and specific cases with $Z_{{\rm m},1}=Z_{{\rm m},2}$ and $Z_{{\rm M},1}=Z_{{\rm M},2}$ are studied (, where $(c_1,c_2,D_{c_1,c_2},X_{m,c_1,c_2})$ is not assumed to be TS). Explicit cases are presented additionally and mainly in \cite{kitazawa10} by explicit real analytic functions $c_i$ and these examples yield MNTS (MNDC) $(c_1,c_2,D_{c_1,c_2},X_{m,c_1,c_2})$. 
\item \label{ex:1.3} We present an NTS (NDC) quadruple $(c_1,c_2,D_{c_1,c_2},X_{m,c_1,c_2})$ which is not minimal. Let $c_1(x):=\sin (x)+\frac{2+\sin {(e^{x^2})}}{4(x^2+1)}$ and let $c_2(x):=\sin (x)+1$. 
The functions $c_1(x)$ and $c_2(x)$ are bounded.
The 1st derivative ${c_1}^{\prime}(x)$ is calculated as ${c_1}^{\prime}(x)=\cos x+\frac{8xe^{x^2}\cos (e^{x^2})\times (x^2+1)-8x(2+\sin {(e^{x^2})})}{16{(x^2+1)}^2}$ and we have sequences $\{x_{+\infty,\pm \infty,j}\}_{j=1}^{\infty}$ ($\{x_{-\infty,\pm \infty,j}\}_{j=1}^{\infty}$) diverging to $+\infty$ (resp. $-\infty$) in the following.
\begin{itemize}
\item The sequence $\{{c_1}^{\prime}(x_{+\infty,+\infty,j})\}_{j=1}^{\infty}$ diverges to $+\infty$ and the sequence $\{{c_1}^{\prime}(x_{+\infty,-\infty,j})\}_{j=1}^{\infty}$ diverges to $-\infty$.
\item The sequence $\{{c_1}^{\prime}(x_{-\infty,+\infty,j})\}_{j=1}^{\infty}$ diverges to $+\infty$ and the sequence $\{{c_1}^{\prime}(x_{-\infty,-\infty,j})\}_{j=1}^{\infty}$ diverges to $-\infty$.
\end{itemize}
Note that this argument on derivatives comes from \cite[Propositions 1 and 2]{kitazawa10} and related arguments. We have a case for $Z_{\rm F}=\{-1,1\}$, we cannot remove any element from $Z_{\rm F}$, and $c_1(S(c_1))\bigcap \{x \leq x_{\rm m}\}$ and $c_1(S(c_1))\bigcap \{x \geq x_{\rm M}\}$ are discrete and closed in $\mathbb{R}-Z_{\rm F}$, where the set $\mathbb{R}-Z_{\rm F} \subset \mathbb{R}$ cannot be larger in this argument on discreteness and closed sets. In this case, each contour of ${\pi}_{2,1} {\mid}_{\overline{D_{c_1,c_2}}}$ is compact and homeomorphic to $D^1$. By the argument of the proof of Theorem \ref{thm:1}, the Reeb digraph is an E-A-digraph.
\end{enumerate}
\end{Ex}
\begin{Thm}
\label{thm:4}
We have the following MNDC quadruples $(c_1,c_2,D_{c_1,c_2},X_{m,c_1,c_2})$ with real analytic functions $c_1$ and $c_2$.
\begin{enumerate}
\item  \label{thm:4.1}
Let $P:\mathbb{R} \rightarrow \mathbb{R}$ be a real analytic function with $P(x) \geq 0$ {\rm (}$x \in \mathbb{R}${\rm )} such that the set $\{x \in \mathbb{R} \mid P(x)=0\}$ is unbounded below.
Suppose that ${\int}_{-\infty}^{\infty} e^{t}P(t)dt$ is also defined.
 Let $c_1(x):=-{\int}_{-\infty}^{x} e^{t}P(t)dt$ and $c_2(x):=c_{2,a}(x):=\frac{x^2}{x^4+1}+a$ with $a \geq 0$. 

This is the case such that $Z_{{\rm m},2}$ and $Z_{{\rm M},2}$ are empty, that $Z_{{\rm m},1}=\{0\}$, and that $Z_{{\rm M},1}$ is empty or a one-point set consisting of exactly one positive number. For example, if $P(x)={\sin}^2 x$, then the value $c_1(x)$ diverges to $-\infty$ as $x$ diverges to $+\infty$ and $Z_{{\rm M},1}$ is empty.

This is a case with exactly one $(c_1,c_2,D_{c_1,c_2},X_{m,c_1,c_2})$-non-normal set $C_0$. Furthermore, the set $C_0$ is mild. This is also of Theorem \ref{thm:3} and the Reeb space is homeomorphic to an E-A-graph. The case $a=0$ gives the Reeb-digraph case of Theorem \ref{thm:2} {\rm (}\ref{thm:2.1}{\rm )}, where the case $a>0$ does not give.
\item \label{thm:4.2}
Let $P:\mathbb{R} \rightarrow \mathbb{R}$ be a real analytic function with $P(x) \geq 0$ {\rm (}$x \in \mathbb{R}${\rm )} such that the set $\{x \in \mathbb{R} \mid P(x)=0\}$ is unbounded below. Suppose that ${\int}_{-\infty}^{\infty} e^{-t^2}P(t)dt$ is also defined. Let $c_1(x):=-{\int}_{-\infty}^{x} e^{-t^2}P(t)dt$ and $c_2(x):=c_{2,a}(x):=\frac{x^2}{x^4+1}+a$ with $a \geq 0$. This is the case such that $Z_{{\rm m},2}$ and $Z_{{\rm M},2}$ are empty, that $Z_{{\rm m},1}=\{0\}$, and that $Z_{{\rm M},1}$ is a one-point set consisting of exactly one positive number. 

This is a case with exactly one $(c_1,c_2,D_{c_1,c_2},X_{m,c_1,c_2})$-non-normal set $C_0$. Furthermore, the set $C_0$ is mild. This is also of Theorem \ref{thm:3} and the Reeb space is homeomorphic to an E-A-graph. The case $a=0$ gives the Reeb-digraph case of Theorem \ref{thm:2} {\rm (}\ref{thm:2.1}{\rm )}, where the case $a>0$ does not give.
\item \label{thm:4.3}
Let $P:\mathbb{R} \rightarrow \mathbb{R}$ be a real analytic function such that $Q(x) \geq 0$ for some real number $t_Q$ and $x \leq t_Q$, that $Q(x) \leq 0$ for $x \geq t_Q$, and that the set $\{x \in \mathbb{R} \mid Q(x)=0\}$ is unbounded. 
Suppose that ${\int}_{-\infty}^{\infty} e^{-t^2}Q(t)dt$ is defined as a non-negative real number. Let $c_1(x):=-{\int}_{-\infty}^{x} e^{-t^2}Q(t)dt$ and $c_2(x):=c_{2,a,b}(x):=\frac{bx^2}{x^4+1}+a$ with $a \geq 0$ and a sufficiently large positive number $b$, respecting the behavior of $Q(x)$ and $a>0$. This is the case such that $Z_{{\rm m},2}$ and $Z_{{\rm M},2}$ are empty, that $Z_{{\rm m},1}=\{0\}$, and that $Z_{{\rm M},1}$ is a one-point set consisting of exactly one non-negative number.

This is a case with exactly one $(c_1,c_2,D_{c_1,c_2},X_{m,c_1,c_2})$-non-normal set $C_0$ or exactly two ones $C_{0}$ and $C_{{\int}_{-\infty}^{\infty} e^{-t^2}Q(t)dt}$, where in the case ${\int}_{-\infty}^{\infty} e^{-t^2}Q(t)dt=0$, the notation of the latter set must be changed. Furthermore, especially, these sets $C_p$ are mild. This is also of Theorem \ref{thm:3} and the Reeb space is homeomorphic to an E-A-graph. The case $a=0$ and ${\int}_{-\infty}^{\infty} e^{-t^2}Q(t)dt=0$ such as the case $Q(x):=x \ {\sin}^2 x$ gives the Reeb-digraph case of Theorem \ref{thm:2} {\rm (}\ref{thm:2.1}{\rm )}, where the other cases {\rm (}such as the case $Q(x):=(x+t_0) {\sin}^2 x$ with a real number $t_0>0${\rm )} do not give.
\end{enumerate}
\end{Thm}

Main ingredients of its proof consist of direct calculations and routine observations. We omit its precise proof.

The following are a kind of our additional remark.
\begin{Rem}
	\label{rem:1}
By our situation, in Theorem \ref{thm:4}, "E-A-graphs" are "E-graphs". For Example \ref{ex:1} (\ref{ex:1.3}), this also holds. 
\end{Rem}
\begin{Rem}
	\label{rem:2}
These "E-graphs" are shown to be graphs by applying arguments in \cite{kitazawa9}, in the case $P(x)={\sin}^2 x$ and $Q(x)=(x+t_1){\sin}^2 x$ with $t_1 \geq  0$ for example, in Theorem \ref{thm:4}. This holds for Example \ref{ex:1} (\ref{ex:1.3}). This is also related to Example \ref{ex:1} (\ref{ex:1.2}).
\end{Rem}
We close this paper by presenting several problems.

Problems similar to Problem \ref{prob:2} are also presented in the end of \cite{kitazawa9} for example and some are solved explicitly in \cite{kitazawa9} and later in \cite{kitazawa10} by controlling derivatives of functions $c:\mathbb{R} \rightarrow \mathbb{R}$ at each infinity explicitly.
\begin{Prob}
	\label{prob:2}
In Theorem \ref{thm:4}, we have presented cases of Theorem \ref{thm:3} such that the Reeb spaces are homeomorphic to E-A-graphs, E-graphs, and graphs (, according to Remark \ref{rem:2}).
\begin{enumerate}
\item In some cases of Theorem \ref{thm:4}, such as the cases presented in Theorem \ref{thm:4} (\ref{thm:4.2}, \ref{thm:4.3}), the functions $c_1$ seem to be real analytic functions which are not elementary functions. Can we obtain similar cases of our Theorem \ref{thm:4} by single elementary functions?, 
\item Can we find examples of real analytic functions for explicit cases for our Theorem \ref{thm:4} we do not investigate in this paper?
\end{enumerate}
\end{Prob}
Problem \ref{prob:3} is related to the structures of Reeb digraphs of smooth functions. Different from cases mainly studied in the present paper, this is on proper smooth functions.
\begin{Prob}
	\label{prob:3}
What is a necessary and sufficient condition for $(R_c,\bar{c})$ to be the Reeb digraph of a smooth function $c:X \rightarrow \mathbb{R}$ on a closed and connected manifold $X$?

 If we restrict the class of $1$-dimensional compact and connected cell complexes to the class of (finite) digraphs, then it is a necessary and sufficient condition for this that the singular value set $c(S(c)) \subset \mathbb{R}$ is a finite set (\cite[Theorem 3.1]{saeki1}).

According to \cite{gelbukh1} (\cite[Theorem 7.5]{gelbukh1}), for a smooth function $c$ on  a closed and connected manifold $X$, $R_c$ is homeomorphic to a so-called ($1$-dimensional) {\it Peano continua} which is metrizable. A {\it Peano continuum} is a compact, connected and locally connected Hausdorff space. In \cite{saeki1, saeki2}, Reeb spaces of smooth functions which are (1-dimensional and metrizable) Peano continua are presented.

In \cite[Theorem 2]{kitazawa7}, the author has presented an ($m_1+m_2$)-dimensional smooth compact and connected submanifold $X_{m_1,m_2}$ of ${\mathbb{R}}^{m_1+m_2+2}$ with no boundary and with positive integers $m_1>0 $ and $m_2>0$ which is the zero set of some smooth function $e_{m_1,m_2}:{\mathbb{R}}^{m_1+m_2+2} \rightarrow \mathbb{R}$ being real analytic outside set $Z_{m_1,m_2} \subset {\mathbb{R}}^{m_1+m_2+2}$ of Lebesgue measure zero. In addition, the author has investigated the Reeb space of the function represented as the restriction ${\pi}_{m+2,1} {\mid}_{X_{m_1,m_2}}$. The Reeb space of this function of the author is shown to be homeomorphic to a Peano-continuum and $(R_{{\pi}_{m+2,1} {\mid}_{X_{m_1,m_2}}},\bar{{\pi}_{m+2,1} {\mid}_{X_{m_1,m_2}}})$ is the Reeb digraph which is also a W-A-digraph and not a digraph, in terminologies of this paper.  
\end{Prob}
 \section{Conflict of interest and Data availability.}
  \noindent {\bf Conflict of interest.} \\
 The author is a researcher at Osaka Central Advanced Mathematical Institute, or a member of OCAMI researchers. This is supported by MEXT Promotion of Distinctive Joint Research Center Program JPMXP0723833165. He thanks this opportunity. He is not employed there. \\
  %Some of works by other researchers and this version may overlap in some of the contents due to the nature that our problems are natural in theory of Morse functions and applications to differential topology and that related mathematical studies are very fundamental and classical in some senses, for example. However the present version of our paper is presented independent of these work. \\
  %Saga Souhatsu Mathematical Seminar (http://inasa.ms.saga-u.ac.jp/Japanese/saga-souhatsu.html), inviting the author as a speaker, is funded and supported by JST Fusion Oriented REsearch for disruptive Science and Technology JPMJFR202U: the author was a speaker on 2024/7/12 supported by this project.\\
  \ \\
  {\bf Data availability.} \\
 No other data is generated. We do not assume non-trivial arguments in preprints being unpublished. To some extent, we may refer to these preprints when we need.

\end{document}